\documentclass[a4paper,11pt,reqno]{amsart}
\usepackage{graphicx, psfrag, amsmath, amscd, amssymb}

\usepackage{ytableau}

\newtheorem{lem}{Lemma}[section]
\newtheorem{thm}[lem]{Theorem}
\newtheorem{pro}[lem]{Proposition}

\newtheorem{exa}[lem]{Example}

\numberwithin{equation}{section}

\newcommand{\ZZ}{{\mathbb{Z}}}

\newcommand{\pet}{{\mathrm{pet}}}
\newcommand{\height}{{\mathrm{ht}}}
\newcommand{\trans}{{\mathrm{c}}}

\newcommand{\core}{{\mathrm{core}}}
\newcommand{\ninv}{{\mathrm{ninv}}}

\newcommand{\G}{{\mathcal{G}}}

\title[Petrie Symmetric Functions]{On Signed Multiplicities of Schur Expansions\\
Surrounding Petrie Symmetric Functions}

\author[Y.-J. Cheng]{Yen-Jen Cheng}
\address{Department of Applied Mathematics,  National Yang Ming Chiao Tung University, Hsinchu 300093, Taiwan, ROC}
\email{yjc7755@gmail.com}

\author[M.-C. Chou]{Meng-Chien Chou}
\address{Department of Mathematics, National Taiwan Normal University, Taipei 116325, Taiwan, ROC}
\email{aoliver466@gmail.com}

\author[S.-P. Eu]{Sen-Peng Eu}
\address{Department of Mathematics, National Taiwan Normal University, Taipei 116325, and Chinese Air Force Academy, Kaohsiung 820009, Taiwan, ROC}
\email{speu@math.ntnu.edu.tw}

\author[T.-S. Fu]{Tung-Shan Fu}
\address{Department of Applied Mathematics, National Pingtung University, Pingtung 900391, Taiwan, ROC}
\email{tsfu@mail.nptu.edu.tw}

\author[J.-C. Yao]{Jyun-Cheng Yao}
\address{Department of Mathematics, National Taiwan Normal University, Taipei 116325, Taiwan, ROC}
\email{0955526287er@gmail.com}

\begin{document}

\subjclass[2020]{05E05, 05A17}
\keywords{Petrie symmetric functions, 
truncated homogeneous symmetric functions, 
modular complete symmetric functions, signed multiplicity free}

\begin{abstract} For $k\ge 1$, the homogeneous symmetric functions $G(k,m)$ of degree $m$ defined by $\sum_{m\ge 0} G(k,m) z^m=\prod_{i\ge 1} \big(1+x_iz+x^2_iz^2+\cdots+x^{k-1}_iz^{k-1}\big)$ are called \emph{Petrie symmetric functions}. As derived by Grinberg and Fu--Mei independently, the expansion of $G(k,m)$ in the basis of Schur functions $s_{\lambda}$ turns out to be signed multiplicity free,  i.e., the coefficients are $-1$, $0$ and $1$. In this paper we give a combinatorial interpretation of the coefficient of $s_{\lambda}$ in terms of the $k$-core of $\lambda$ and a sequence of rim hooks of size $k$ removed from $\lambda$. We further study the product of $G(k,m)$ with a power sum symmetric function $p_n$. For all $n\ge 1$, we give necessary and sufficient conditions on the parameters $k$ and $m$ in order for the expansion of $G(k,m)\cdot p_n$ in the basis of Schur functions to be signed multiplicity free. This settles affirmatively a conjecture of Alexandersson as the special case $n=2$.
\end{abstract}

\maketitle

\section{Introduction} 

\subsection{Preliminaries and background}
A \emph{partition} $\lambda$ of $m$ is a sequence $(\lambda_1,\lambda_2,\dots,\lambda_{\ell})$ of integers such that $\lambda_1\ge\lambda_2\ge\cdots\ge\lambda_{\ell}>0$ and $\lambda_1+\lambda_2+\cdots+\lambda_{\ell}=m$. Each entry $\lambda_i$ is a \emph{part} of $\lambda$. We write $\lambda\vdash m$, $|\lambda|=m$ (\emph{size} of $\lambda$) and $\ell(\lambda)=\ell$ (\emph{length} of $\lambda$). Sometimes we write $(\lambda^{d_1}_1,\lambda^{d_2}_2,\cdots)$ if there are $d_1$ parts of $\lambda_1$ and $d_2$ parts of $\lambda_2<\lambda_1$ etc.  The \emph{conjugate} of $\lambda$, denoted $\lambda^{\trans}$, is the partition $(\lambda^{\trans}_1,\lambda^{\trans}_2,\dots,\lambda^{\trans}_{\lambda_1})$ defined by $\lambda^{\trans}_i=\#\{j: \lambda_j \ge i\}$. The \emph{Young diagram} $[\lambda]$ of a partition $\lambda$ is a left-justified array of cells with $\lambda_i$ cells in the $i$th row.
For partitions $\lambda,\mu$ we say that $\mu\subset \lambda$ if $\mu_j\le \lambda_j$ for all $j$. The diagram $[\lambda/\mu]$ of a \emph{skew shape} $\lambda/\mu$ is the array obtained by removing $[\mu]$ from the north-west corner of $[\lambda]$, where $\mu\subset \lambda$. We say that $\lambda/\mu$ is a \emph{rim hook} (or \emph{border strip} or \emph{ribbon}) if $[\lambda/\mu]$ is a connected diagram with no $2\times 2$ square. The \emph{height} $\height(\lambda/\mu)$ of a rim hook is defined to be one less than the number of rows of $[\lambda/\mu]$. Let $\core_k(\lambda)$ denote the $k$-\emph{core} obtained from $\lambda$ by successively removing all rim hooks of size $k$. We make use of the $\chi$-notation that maps each condition $C$ onto $\{0, 1\}$, defined as $\chi(C)=1$ if $C$ is true, and 0 otherwise. For a positive integer $k$ and a partition $\lambda$, the $k$-\emph{Petrie number} $\pet_k(\lambda)$ of $\lambda$ is the integer defined by
\begin{equation} \label{eqn:k-Petrie}
\pet_k(\lambda):=\det\big[\chi(0\le\lambda_i-i+j<k  )\big]_{i,j=1}^{\ell(\lambda)}.
\end{equation}
By a classical result of Gordon--Wilkinson \cite{GW}, the $k$-Petrie numbers have only three possibilities $\{-1,0,1\}$. For example, if $\lambda=(3,3,1)$ then
\[
\pet_3(\lambda)=\det
\begin{bmatrix}
       0  &   0  &  0  \\
       1  &   0  &  0  \\
       0  &   1  &  1  \\
\end{bmatrix}=0 \qquad\mbox{and}\qquad
\pet_4(\lambda)=\det
\begin{bmatrix}
       1  &   0  &  0  \\
       1  &   1  &  0  \\
       0  &   1  &  1  \\
\end{bmatrix}=1.
\]

Recently, Grinberg \cite{Grinberg} and Fu--Mei \cite{FM} studied independently the homogeneous symmetric functions $G(k,m)$ of degree $m$ defined by
\begin{equation}
\sum_{m\ge 0} G(k,m) z^m=\prod_{i\ge 1} \big(1+x_iz+x^2_iz^2+\cdots+x^{k-1}_iz^{k-1}\big).
\end{equation}
Following \cite{Grinberg}, the symmetric functions $G(k,m)$ are called \emph{Petrie symmetric functions} (or \emph{truncated homogeneous symmetric functions} in \cite{FM}). First appeared in \cite{Doty-Walker}, they are also known as \emph{modular complete symmetric functions}. Expressed in terms of monomial symmetric functions, the Petrie symmetric functions can be written as
\begin{equation} \label{eqn:monomial}
G(k,m)=\sum_{\substack{\lambda\,\vdash\, m; \\ \lambda_i<k\mbox{ \footnotesize for all $i$}}} m_{\lambda}.
\end{equation}
In particular, $G(2,m)=e_m$ and $G(k,m)=h_m$ if $k>m$. Grinberg \cite{Grinberg} found that the coefficients of the expansion of $G(k,m)$ in the basis of Schur functions turn out to be $k$-Petrie numbers (see Theorem \ref{thm:k-Petrie-Schur}) and gave an explicit description for the coefficients (see Theorem \ref{thm:Grinberg}).

\begin{thm}[Grinberg] \label{thm:k-Petrie-Schur} For $k\ge1$ and $m\ge 0$, we have
\begin{equation} \label{eqn:coefficient}
G(k,m)=\sum_{\lambda\,\vdash\, m} \pet_k(\lambda) s_{\lambda}.
\end{equation}
\end{thm}
An equivalent description for $\pet_k(\lambda)$ was given by Fu--Mei \cite[Proposition~2.9]{FM}. 
For $\lambda,\mu\vdash m$, we say that $\lambda\unrhd\mu$ if $\lambda_1+\cdots+\lambda_i\ge\mu_1+\cdots+\mu_i$ for all $i$. By careful analysis and using a skewing operator on symmetric functions, Grinberg \cite[Section~4]{Grinberg} proved the following conjecture of Liu--Polo \cite[Corollary~1.4.4]{LP}, who obtained the result for all primes $k$.

\begin{thm}[Liu--Polo conjecture] \label{thm:LP-conjecture} For $k\ge 2$, we have
\[
\sum_{\substack{\lambda\,\vdash\, k; \\ \lambda\unlhd (k-1,1)}} m_{\lambda}=\sum_{i=0}^{k-2} (-1)^i s_{(k-1-i,1^{i+1})}. 
\]
and 
\[
\sum_{\substack{\lambda\,\vdash\, 2k-1; \\ \lambda\unlhd (k-1,k-1,1)}} m_{\lambda}=\sum_{i=0}^{k-2} (-1)^i s_{(k-1,k-1-i,1^{i+1})}. 
\]
\end{thm}

\subsection{Main results} 
The first main result in this paper is the following interpretation of $\pet_k(\lambda)$ in terms of the $k$-core of $\lambda$ and a sequence of rim hooks of size $k$ removed from $\lambda$. As an immediate result, we give a simple proof of the Liu--Polo conjecture.

\begin{thm} \label{thm:main-1} For $k\ge1$ and $\lambda\vdash m$, the following results hold.
\begin{enumerate}
\item If $\lambda_1\ge k$ then $\pet_k(\lambda)=0$.
\item Let $\lambda_1<k$. If $\core_k(\lambda)$ has more than one part then $\pet_k(\lambda)=0$; otherwise let $q=\lfloor\frac{m}{k}\rfloor$ and let $\lambda^0,\lambda^1,\dots,\lambda^q$ be any sequence of partitions such that
\[
\core_k(\lambda)=\lambda^0\subset\lambda^1\subset\cdots\subset\lambda^q=\lambda
\]
and $\lambda^{j}/\lambda^{j-1}$ is a rim hook of size $k$ for each $j\in\{1,\dots, q\}$. Then 
\[
\pet_k(\lambda)=\left\{
                   \begin{array}{ll}
                   1       & \mbox{if $\lambda=\core_k(\lambda)$} \\[0.5ex]
                   {\displaystyle\prod_{j=1}^q (-1)^{\height(\lambda^{j}/\lambda^{j-1})+1}} & \mbox{otherwise.} 
                   \end{array}
                \right.
\]
\end{enumerate}
\end{thm}

\begin{exa} \label{exa:G58} {\rm
Let $k=4$ and $m=8$. The partitions $\lambda\vdash 8$ with $\lambda_1<4$ and empty 4-core are shown in Figure \ref{fig:4-core} along with their rim hooks of size 4. The values of $\height(\lambda^1/\lambda^0)+\height(\lambda/\lambda^1)+2$ for these partitions are 4, 5, 6, 6, 7 and 8, accordingly. We have $G(4,8)=s_{(3,3,2)}-s_{(3,2,2,1)}+s_{(3,1,1,1,1,1)}+s_{(2,2,2,2)}-s_{(2,1,1,1,1,1,1)}+s_{(1,1,1,1,1,1,1,1)}$.

}
\end{exa}

\begin{figure}[ht]
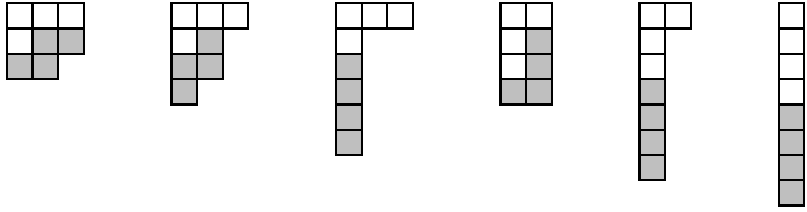

\begin{center}
\ytableausetup{smalltableaux}
\ydiagram{3,3,2}*[*(lightgray)]{0,1+2,2}\hspace{0.4in}
\ydiagram{3,2,2,1}*[*(lightgray)]{0,1+1,2,1} \hspace{0.4in}
\ydiagram{3,1,1,1,1,1}*[*(lightgray)]{0,0,1,1,1,1} 
\hspace{0.4in}
\ydiagram{2,2,2,2}*[*(lightgray)]{0,1+1,1+1,2}
\hspace{0.4in}
\ydiagram{2,1,1,1,1,1,1}*[*(lightgray)]{0,0,0,1,1,1,1}
\hspace{0.4in}
\ydiagram{1,1,1,1,1,1,1,1}*[*(lightgray)]{0,0,0,0,1,1,1,1}
\end{center}
\caption{\small The set of partitions $\lambda\vdash 8$ with $\lambda_1<4$ and empty 4-core.}
\label{fig:4-core}
\end{figure}

We further study the signed multiplicities of the  expansion of the product of $G(k,m)$ with a power sum symmetric function $p_n$ in the Schur basis. For $n=2$, Alexandersson \cite[Conjecture~5.1]{Grinberg} suggested a conjecture that the expansion of $G(k,m)\cdot p_2$ is signed multiplicity free, i.e., the coefficients are $-1$, $0$ and $1$. However, it is false if $p_2$ is replaced by $p_3$. For example,
\begin{align*}
G(3,5)\cdot p_2 &= s_{(4,2,1)}-s_{(4,1,1,1)}-s_{(3,3,1)}+s_{(2,2,2,1)}-s_{(2,2,1,1,1)}+s_{(2,1,1,1,1,1)} \\
G(3,5)\cdot p_3 &= s_{(5,2,1)}-s_{(5,1,1,1)}-s_{(4,3,1)}+s_{(3,3,1,1)}-2s_{(2,2,2,2)}+s_{(2,2,1,1,1,1)}-s_{(2,1,1,1,1,1,1)}.
\end{align*}

The second main result in this paper is to give necessary and sufficient conditions on the parameters $k$ and $m$ in order for such expansions of $G(k,m)\cdot p_n$ to be signed multiplicity free for all $n\ge 1$.

\begin{thm} \label{thm:main-2} For all $n\ge 1$, in the following expansion
\[
G(k,m)\cdot p_n=\sum_{\lambda\,\vdash\, (m+n)} g_{\lambda}s_{\lambda},
\]
there exists a $s_{\mu}$ with coefficient $g_{\mu}\not\in\{-1,0,1\}$
if and only if $k\ge 3$, $k\mid n$ and $m\ge n$.

\end{thm}

\section{Proof of Theorem \ref{thm:main-1}}
In this section we shall study the connection between the $k$-core and the $k$-Petrie number of a partition $\lambda$, and prove Theorem \ref{thm:main-1} using Grinberg's description of $\pet_k(\lambda)$.

\smallskip
Fix a positive integer $k$. For any partition $\mu=(\mu_1,\dots,\mu_{\ell})$ with $\ell<k$, we will identify $\mu$ with the $(k-1)$-tuple $(\mu_1,\dots,\mu_{k-1})$, where $\mu_{\ell+1}=\mu_{\ell+2}=\cdots=0$. For the congruence modulo $k$, let $\overline{j}$ denote the integer congruent to $j$ mod $k$ with $1\le \overline{j}\le k$ for all $j\in\ZZ$.
Define a map $\beta:\mu\mapsto (\beta_1,\dots,\beta_{k-1})\in\ZZ^{k-1}$ and a map $\gamma:\mu\mapsto(\gamma_1,\dots,\gamma_{k-1})\in\{1,\dots,k\}^{k-1}$ by setting
\begin{equation} \label{eqn:beta-gamma}
\beta_i=\mu_i-i\qquad\mbox{and}\qquad \gamma_i=\overline{\beta_i}\pmod k
\end{equation}
for each $i\in\{1,\dots, k-1\}$. Moreover, if all of $\gamma_1,\dots,\gamma_{k-1}$ are distinct, let $\ninv(\gamma(\mu))$ denote the number of \emph{non-inversions} of $(\gamma_1,\dots,\gamma_{k-1})$, i.e.,
\begin{equation} \label{eqn:ninv}
\ninv(\gamma(\mu)):=\#\{(i,j):\gamma_i<\gamma_j, 1\le i<j\le k-1 \}.
\end{equation}
For example, let $k=6$ and $\mu=(7,4,2,1)$. We identify $\mu$ with $(7,4,2,1,0)$, and find $\beta(\mu)=(6,2,-1,-3,-5)$ and $\gamma(\mu)=(6,2,5,3,1)$. We have $\ninv(\gamma(\mu))=2$. 

\smallskip

Grinberg derived the following result as a combinatorial alternative to (\ref{eqn:k-Petrie}) for the calculation of $\pet_k(\lambda)$   \cite[Theorem~2.15]{Grinberg}.

\begin{thm}[Grinberg] \label{thm:Grinberg} For all $k\ge 1$ and any partition $\lambda$, the following results hold.
\begin{enumerate}
\item If $\lambda_1\ge k$ then $\pet_k(\lambda)=0$.
\item Let $\lambda_1<k$. Let $\beta(\lambda^{\trans})=(\beta_1,\dots,\beta_{k-1})$ and $\gamma(\lambda^{\trans})=(\gamma_1,\dots,\gamma_{k-1})$. If all of $\gamma_1,\dots,\gamma_{k-1}$ are distinct then
\begin{equation} \label{eqn:pet-calculation}
\pet_k(\lambda)=(-1)^{\beta_1+\cdots+\beta_{k-1}+\ninv(\gamma(\lambda^{\trans}))+\gamma_1+\cdots+\gamma_{k-1}};
\end{equation}
otherwise $\pet_k(\lambda)=0$.
\end{enumerate}
\end{thm}

Let $\delta_n:=(n-1,n-2,\dots,0)$. We will encode a partition $\mu$ with $\ell(\mu)<k$ in an \emph{abacus} with $k$ vertical runners labelled by $0,1,\dots,k-1$. The partition $\mu$ is shown in the abacus by a set of $k-1$ beads, described by the strictly decreasing sequence
\begin{equation} \label{eqn:alpha}
\mu+\delta_{k-1}:=(\mu_1+k-2, \mu_2+k-3,\dots, \mu_{k-1})
\end{equation}
(these are called the $\beta$-\emph{numbers} for $\mu$ in \cite{JK,Walker}).
For each $j$, we place a bead representing $\mu_j$ in the $r_j$-th row and $c_j$-th column if $\mu_j+k-1-j=k(r_j-1)+c_j$, $r_j\ge 1$ and $0\le c_j\le k-1$. For example, if $k=6$ and $\mu=(7,4,2,1)$, the corresponding abacus of $\mu$ is shown in Figure \ref{fig:abacus}, which is described by the sequence $\mu+\delta_{5}=(11,7,4,2,0)$. In this notation, removal of a rim hook of size $k$ corresponds to moving a bead one place up in its runner \cite[Lemma 2.7.13]{JK}.

\begin{figure}[ht]
\begin{center}
\begin{tabular}{cccccc}   
0  &  1 & 2 & 3 & 4 &  5 \\
    \hline
   $\circ$  & $\cdot$  & $\circ$  & $\cdot$    &  $\circ$  &  $\cdot$   \\
   $\cdot$  & $\circ$  & $\cdot$  & $\cdot$    &  $\cdot$  &  $\circ$     
\end{tabular}
\end{center}
\caption{\small The abacus with six runners whose beads encode the partition $(7,4,2,1)$.}
\label{fig:abacus}
\end{figure}


\smallskip
To prove Theorem \ref{thm:main-1}, it suffices to consider the partitions $\lambda$ with $\lambda_1<k$. We have the following relation between $\core_k(\lambda)$ and $\gamma\big(\lambda^{\trans}\big)$.

\begin{lem} \label{lem:one-part} For any partition $\lambda$ with $\lambda_1<k$, let $\gamma\big(\lambda^{\trans}\big)=(\gamma_1,\dots,\gamma_{k-1})$. Then all of $\gamma_1,\dots,\gamma_{k-1}$ are distinct if and only if the $k$-core of $\lambda$ either has only one part or is empty. 
\end{lem}

\begin{proof} 
Encoding $\lambda^{\trans}$ in an abacus, $\lambda^{\trans}+\delta_{k-1}=(\lambda^{\trans}_1+k-2,\lambda^{\trans}_2+k-3,\dots,\lambda^{\trans}_{k-1})$. By (\ref{eqn:beta-gamma}) and (\ref{eqn:alpha}), we have
\begin{equation} \label{eqn:gamma_j-1}
\lambda^{\trans}_j+k-1-j\equiv\gamma_j-1\pmod k
\end{equation}
for each $j\in\{1,\dots, k-1\}$. 
Notice that the bead representing $\lambda^{\trans}_j$ is in the runner labelled by $\gamma_j-1$. Hence all of $\gamma_1,\dots,\gamma_{k-1}$ are distinct if and only if these $k-1$ beads that encode $\lambda^{\trans}$ are in distinct runners of the abacus. Moving each bead to the top of its runner, the $k$-core of $\lambda^{\trans}$ is encoded in the abacus by the sequence
\begin{equation} \label{eqn:core+delta}
\core_k\big(\lambda^{\trans}\big)+\delta_{k-1}=(k-1,\dots,k-d,k-d-2,\dots,0)
\end{equation}
for some $d\in\{1,\dots,k-1\}$ or $(k-2,k-3,\dots,0)$, where $d$ is the remainder of $|\lambda|$ divided by $k$. This implies $\core_k\big(\lambda^{\trans}\big)=(1^d)$ or $(0)$, and hence $\core_k(\lambda)$ either has only one part or is empty. The converse is straightforward.
\end{proof}

By (\ref{eqn:gamma_j-1}) and (\ref{eqn:core+delta}), note that if all of $\gamma_1,\dots,\gamma_{k-1}$ are distinct, the underlying set $\{\gamma_1,\dots,\gamma_{k-1}\}$ of $\gamma\big(\lambda^{\trans}\big)$ is $\{1,\dots,k\}\setminus \{k-d\}$. The following result specifically shows that
when a rim hook $\lambda/\mu$ of size $k$ is removed from $\lambda$, the sequence $\gamma\big(\mu^{\trans}\big)$ is obtained by cyclically shifting a portion of $\gamma\big(\lambda^{\trans}\big)$.

\begin{lem} \label{lem:non-inversion} For any partition $\lambda$ with $\lambda_1<k$, let $\beta(\lambda^{\trans})=(\beta_1,\dots,\beta_{k-1})$ and let $\gamma(\lambda^{\trans})=(\gamma_1,\dots,\gamma_{k-1})$. 
Let $\mu$ be a partition such that $\mu\subset\lambda$, $\lambda/\mu$ is a rim hook of size $k$, and the diagram $[\lambda/\mu]$ has $b$ columns. Then for some $i\in\{1,\dots,k-b\}$ we have
\begin{equation} \label{eqn:beta-mu}
\beta\big(\mu^{\trans}\big)=(\beta_1,\dots,\beta_{i-1},\beta_{i+1},\dots,\beta_{i+b-1},\beta_{i}-k,\beta_{i+b},\dots,\beta_{k-1})
\end{equation}
and
\begin{equation} \label{eqn:gamma-mu}
\gamma\big(\mu^{\trans}\big)=
\begin{pmatrix}
\gamma_i &  \gamma_{i+1} &\cdots & \gamma_{i+b-2} & \gamma_{i+b-1} \\
\gamma_{i+1} & \gamma_{i+2} & \cdots & \gamma_{i+b-1} & \gamma_{i}
\end{pmatrix}
\gamma\big(\lambda^{\trans}\big).
\end{equation}
Moreover, if all of $\gamma_1,\dots,\gamma_{k-1}$ are distinct then 
\begin{equation} \label{eqn:parity}
\ninv\big(\gamma(\lambda^{\trans})\big)+\ninv\big(\gamma(\mu^{\trans})\big)\equiv k+\height\big(\lambda/\mu\big)+1 \pmod 2.
\end{equation} 
\end{lem}

\begin{proof} Let $\xi$ denote the rim hook $\lambda/\mu$. Note that $[\xi^{\trans}]$ has $b$ rows. Suppose $[\xi^{\trans}]$ meets $[\lambda^{\trans}]$ from row $i$ to row $i+b-1$. 
We observe that
\begin{equation} \label{eqn:mu-transpose}
\mu^{\trans}_j=
\begin{cases}
\lambda^{\trans}_j &  1\le j\le i-1\\[1ex]
\lambda^{\trans}_{j+1}-1 &  i\le j\le i+b-2\\[1ex]
\lambda^{\trans}_j & i+b\le j\le k-1.
\end{cases}
\end{equation}
Since $|\mu^{\trans}|=|\lambda^{\trans}|-k$, we have
\begin{equation} \label{eqn:mu_i+b-1}
\mu^{\trans}_{i+b-1}=\lambda^{\trans}_{i}-k+b-1.
\end{equation}
Let $\beta\big(\mu^{\trans}\big)=\big(\beta'_1,\dots,\beta'_{k-1}\big)$. By (\ref{eqn:beta-gamma}), (\ref{eqn:mu-transpose}) and (\ref{eqn:mu_i+b-1}), we have
\begin{equation} \label{eqn:beta-mu-transpose}
\beta'_j=
\begin{cases}
\beta_j &  1\le j\le i-1\\
\beta_{j+1} & i\le j\le i+b-2\\
\beta_{i}-k & j=i+b-1 \\
\beta_j & i+b\le j\le k-1.
\end{cases}
\end{equation}
Thus (\ref{eqn:beta-mu}) holds. Creating $\gamma\big(\mu^{\trans}\big)$ from $\beta\big(\mu^{\trans}\big)$, we have
\[
\gamma\big(\mu^{\trans}\big)=(\gamma_1,\dots,\gamma_{i-1},\gamma_{i+1},\dots,\gamma_{i+b-1},\gamma_{i},\gamma_{i+b},\dots,\gamma_{k-1}),
\]
which is the result obtained by acting on $\gamma\big(\lambda^{\trans}\big)$ a $b$-cycle that cyclically shifts $\gamma_i,\gamma_{i+1},\dots,\gamma_{i+b-1}$.  Moreover, if all of $\gamma_1,\dots,\gamma_{k-1}$ are distinct then $\ninv\big(\gamma(\lambda^{\trans})\big)+\ninv\big(\gamma(\mu^{\trans})\big)\equiv b-1 \pmod 2$. Using the facts $\height(\xi)+\height\big(\xi^{\trans}\big)=k-1$ and $\height\big(\xi^{\trans}\big)=b-1$, the result follows.
\end{proof}

\begin{exa} \label{exa:cyclic-shift} {\rm
Let $k=6$. Consider the partition $\lambda=(4,3,2,2,1,1,1)$, where $\lambda^{\trans}=(7,4,2,1)$, $\beta(\lambda^{\trans})=(6,2,-1,-3,-5)$ and $\gamma(\lambda^{\trans})=(6,2,5,3,1)$. There are two ways to remove a rim hook of size 6 from $\lambda$, as shown in Figure \ref{fig:rim-hook-6}.

(i)  Let $\mu=(2,1,1,1,1,1,1)$. Note that $\lambda/\mu$ is a rim hook with three columns and $\height(\lambda/\mu)=3$. We have $\mu^c=(7,1)$, $\beta(\mu^{\trans})=(6,-1,-3,-4,-5)$ and $\gamma(\mu^{\trans})=(6,5,3,2,1)$. Moreover,  $\ninv(\gamma\big(\mu^{\trans})\big)=\ninv(\gamma\big(\lambda^{\trans})\big)-2$.

(ii) Let $\mu=(4,3,1)$. Note that $\lambda/\mu$ is a rim hook with two columns and $\height(\lambda/\mu)=4$. We have $\mu^c=(3,2,2,1)$, $\beta(\mu^{\trans})=(2,0,-1,-3,-5)$ and $\gamma(\mu^{\trans})=(2,6,5,3,1)$. Moreover,  $\ninv(\gamma\big(\mu^{\trans})\big)=\ninv(\gamma\big(\lambda^{\trans})\big)+1$.

}
\end{exa}

\begin{figure}[ht]
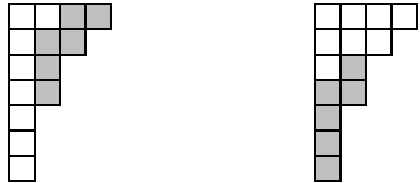

\begin{center}
\ytableausetup{smalltableaux}
\ydiagram{4,3,2,2,1,1,1}*[*(lightgray)]{2+2,1+2,1+1,1+1,1+0,1+0,1+0} \hspace{1in}
\ydiagram{4,3,2,2,1,1,1}*[*(lightgray)]{4+0,3+0,1+1,0+2,0+1,0+1,0+1} \hspace{1in}
\end{center}
\caption{\small The two ways to remove a rim hook of size 6 from $\lambda=(4,3,2,2,1,1,1)$.}
\label{fig:rim-hook-6}
\end{figure}

Now we prove Theorem \ref{thm:main-1}.

\medskip
\noindent
\emph{Proof of Theorem \ref{thm:main-1}.}
Let $\lambda\vdash m$ with $\lambda_1<k$. By Theorem \ref{thm:Grinberg}(ii) and Lemma \ref{lem:one-part}, if $\core_k(\lambda)$ has more than one part then $\pet_k(\lambda)=0$. Suppose $\core_k(\lambda)=(d)$, where $d\equiv m\pmod k$ and $0\le d\le k-1$.
Consider a sequence $\lambda^0,\lambda^1,\dots,\lambda^q$ of partitions such that
\[
\core_k(\lambda)=\lambda^0\subset\lambda^1\subset\cdots\subset\lambda^q=\lambda
\]
and $\lambda^{i}/\lambda^{i-1}$ is a rim hook of size $k$ for each $i\in\{1,\dots,q\}$.

For the $k$-core of $\lambda$, we observe that $\beta\big({\lambda^0}^{\trans}\big)=(0,-1,\dots,-d+1,-d-1,\dots,-k+1)$ if $d\ge 1$ or  $(-1,-2,\dots,-k+1)$ if $d=0$. Then $\gamma\big({\lambda^0}^{\trans}\big)=(k,k-1,\dots,k-d+1,k-d-1,\dots,1)$ or $(k-1,k-2,\dots,1)$, and hence $\ninv\big(\gamma({\lambda^0}^{\trans})\big)=0$. By (\ref{eqn:pet-calculation}), we have $\pet_k(\lambda^0)=(-1)^{2d}=1$.

Given $j\in\{1,\dots, q\}$, let $\beta\big({\lambda^j}^{\trans}\big)=(\beta_1,\dots,\beta_{k-1})$, $\gamma\big({\lambda^j}^{\trans}\big)=(\gamma_1,\dots,\gamma_{k-1})$, $\beta\big({\lambda^{j-1}}^{\trans}\big)=(\beta'_1,\dots,\beta'_{k-1})$, and $\gamma\big({\lambda^{j-1}}^{\trans}\big)=(\gamma'_1,\dots,\gamma'_{k-1})$. By Lemma \ref{lem:non-inversion}, we have
\begin{equation} \label{eqn:beta-gamma-k}
\begin{aligned}
\beta'_1+\cdots+\beta'_{k-1}&=\beta_1+\cdots+\beta_{k-1}-k, \\
\gamma'_1+\cdots+\gamma'_{k-1}&=\gamma_1+\cdots+\gamma_{k-1},
\end{aligned}
\end{equation}
and
\begin{equation} \label{eqn:parity-2}
\ninv\big(\gamma\big({\lambda^j}^{\trans}\big)\big)+\ninv\big(\gamma\big({\lambda^{j-1}}^{\trans}\big)\big)\equiv k+\height(\lambda^{j}/\lambda^{j-1})+1 \pmod 2.
\end{equation}
By (\ref{eqn:pet-calculation}), (\ref{eqn:beta-gamma-k}) and (\ref{eqn:parity-2}), we have
\[
\frac{\pet_k\big(\lambda^j\big)}{\pet_k\big(\lambda^{j-1}\big)}=\frac{(-1)^{\beta_1+\cdots+\beta_{k-1}+\ninv(\gamma({\lambda^j}^{\trans}))+\gamma_1+\cdots+\gamma_{k-1}}}{(-1)^{\beta'_1+\cdots+\beta'_{k-1}+\ninv(\gamma({\lambda^{j-1}}^{\trans}))+\gamma'_1+\cdots+\gamma'_{k-1}}}=(-1)^{\height(\lambda^{j}/\lambda^{j-1})+1}.
\]
By iteration, it follows that
\[
\pet_k(\lambda)=\prod_{j=1}^q \frac{\pet_k\big(\lambda^j\big)}{\pet_k\big(\lambda^{j-1}\big)}=\prod_{j=1}^q (-1)^{\height(\lambda^{j}/\lambda^{j-1})+1}.
\]
The proof of Theorem \ref{thm:main-1} is completed.
\qed

\smallskip

We now give a combinatorial proof of the Liu--Polo conjecture.

\medskip
\noindent
\emph{Proof of Theorem \ref{thm:LP-conjecture}.} Note that the partitions $\lambda\vdash k$ ($2k-1$, respectively) such that $\lambda\unlhd (k-1,1)$ ($(k-1,k-1,1)$, respectively) are precisely those $\lambda$  with $\lambda_i<k$ for all $i$. By (\ref{eqn:monomial}), we have
\[
\sum_{\substack{\lambda\,\vdash\, k; \\ \lambda\unlhd (k-1,1)}} m_{\lambda}=G(k,k)\qquad\mbox{and}\qquad\sum_{\substack{\lambda\,\vdash\, 2k-1; \\ \lambda\unlhd (k-1,k-1,1)}} m_{\lambda}=G(k,2k-1).
\]
Moreover, such partitions $\lambda$ with $\core_k(\lambda)=(0)$ ($(k-1)$, respectively) are of the form $\lambda=(k-1-i,1^{i+1})$ ($(k-1,k-1-i,1^{i+1})$, respectively) for some $i\in\{0,1,\dots,k-2\}$, containing a unique rim hook $\xi=(k-1-i,1^{i+1})$ of size $k$, where $\height(\xi)=i+1$. By Theorems \ref{thm:k-Petrie-Schur} and \ref{thm:main-1}, the results follow. 
\qed

\section{Proof of Theorem \ref{thm:main-2}}

In this section we shall study the expansion of the product of $G(k,m)$ with a power sum symmetric function $p_n$ in the basis of Schur functions, and prove Theorem \ref{thm:main-2}.

\smallskip
The \emph{Murnaghan--Nakayama rule} gives an expression for the product of a Schur function with a power sum symmetric function:
\begin{equation} \label{eqn:MN-rule}
s_{\lambda}p_n=\sum_{\lambda^{+}} (-1)^{\height(\lambda^{+}/\lambda)} s_{\lambda^{+}},
\end{equation}
summed over all partitions $\lambda^{+}\supset\lambda$ such that $\lambda^{+}/\lambda$ is a rim hook of size $n$. See \cite[Section~7.17]{EC2}.

\smallskip
Define
\[
\G(k,m):=\{\lambda\vdash m: \pet_k(\lambda)\neq 0 \}.
\]
By the proof of Theorem \ref{thm:main-1}, recall that every $\lambda\in\G(k,m)$ has the following properties:
\begin{enumerate}
\item $\lambda_1<k$.
\item $\core_k(\lambda)=(d)$, where $d\equiv m\pmod k$ and $0\le d\le k-1$. 
\item The underlying set of $\gamma(\lambda^{\trans})$ is $\{1,\dots,k\}\setminus\{k-d\}$.
\end{enumerate}

Using Murnaghan--Nakayama rule and Theorem \ref{thm:k-Petrie-Schur}, we have 
\begin{equation} \label{eqn:Schur-expansion}
G(k,m)\cdot p_n=\sum_{\lambda\in\G(k,m)}\sum_{\lambda^{+}} (-1)^{\height(\lambda^{+}/\lambda)}\pet_k(\lambda)s_{\lambda^{+}},
\end{equation}
where the inner sum is over all $\lambda^{+}$ such that $\lambda^{+}/\lambda$ is a rim hook of size $n$. We shall enumerate the coefficient of $s_{\lambda^{+}}$ in the expression (\ref{eqn:Schur-expansion}).

\begin{exa} {\rm
We have $G(5,8)=s_{(4,4)}-s_{(3,3,1,1)}+s_{(3,2,1,1,1)}-s_{(3,1,1,1,1,1)}$. The  expansion of $G(5,8)\cdot p_3$ in the Schur basis is shown below. Notice that the term $s_{(4,4,1,1,1)}$ vanishes since the two occurrences of $s_{(4,4,1,1,1)}$ have the opposite signs.
\begin{align*}
G(5,8)\cdot p_3 &= s_{(4,4)}p_3-s_{(3,3,1,1)}p_3+s_{(3,2,1,1,1)}p_3-s_{(3,1,1,1,1,1)}p_3 \\
&= \big(s_{(7,4)}-s_{(6,5)}+s_{(4,4,3)}-s_{(4,4,2,1)}+s_{(4,4,1,1,1)}\big) \\
&\qquad - \big(s_{(6,3,1,1)}-s_{(5,4,1,1)}-s_{(3,3,3,2)}+s_{(3,3,1,1,1,1,1)}\big) \\
&\qquad + \big(s_{(6,2,1,1,1)}-s_{(4,4,1,1,1)}-s_{(3,3,3,1,1)}+s_{(3,2,2,2,2)}+s_{(3,2,1,1,1,1,1,1)}\big) \\
&\qquad - \big(s_{(6,1,1,1,1,1)}-s_{(3,3,2,1,1,1)}+s_{(3,2,2,2,1,1)}+s_{(3,1,1,1,1,1,1,1,1)}\big).
\end{align*}

}
\end{exa}

\begin{pro} \label{pro:k=2} For any $\lambda,\mu\in\G(k,m)$, let $\lambda^{+},\mu^{+}$ be  partitions such that $\lambda^{+}/\lambda$ and $\mu^{+}/\mu$ are rim hooks of size $n$ with $\mu^{+}=\lambda^{+}$. If $\lambda^{+}_1\ge k$ then $\mu=\lambda$.
\end{pro}

\begin{proof}  If $\lambda^{+}_1\ge k$, we observe that there is a unique way to remove a rim hook of size $n$ from $\lambda^{+}$ such that the first part of resulting partition $\lambda$ is less than $k$. Specifically, $\lambda_j=\lambda^{+}_{j+1}-1$ for $1\le j\le \ell-1$, $\lambda_{\ell}=\lambda^{+}_1-n+\ell-1$, and $\lambda_j=\lambda^{+}_j$ for $j\ge \ell+1$, where $\ell$ is the number of rows of $[\lambda^{+}/\lambda]$. Hence $\mu=\lambda$.
\end{proof}

The following result shows the relations between the sequences $\gamma\big({\lambda^{+}}^{\trans}\big)$ and $\gamma\big(\lambda^{\trans}\big)$. 

\begin{lem} \label{lem:cyclic-shift}
For any $\lambda\in\G(k,m)$, let $\beta(\lambda^{\trans})=(\beta_1,\dots,\beta_{k-1})$ and let $\gamma(\lambda^{\trans})=(\gamma_1,\dots,\gamma_{k-1})$. Let $\lambda^{+}$ be a partition such that $\lambda^{+}/\lambda$ is a rim hook of size $n$, $\lambda^{+}_1<k$ and the diagram $[\lambda^{+}/\lambda]$ has $a$ columns. Then for some $j\in\{1,...,k-a\}$ we have
\begin{equation} \label{eqn:beta-lambda+c}
\beta\big({\lambda^{+}}^{\trans}\big)=(\beta_1,\dots,\beta_{j-1},\beta_{j+a-1}+n,\beta_{j},\dots,\beta_{j+a-2},\beta_{j+a},\dots,\beta_{k-1})
\end{equation}
and
\begin{equation} \label{eqn:gamma-lambda-star}
\gamma\big({\lambda^{+}}^{\trans}\big) =
\begin{pmatrix}
\gamma_j &  \gamma_{j+1} & \gamma_{j+2} & \cdots & \gamma_{j+a-1} \\
\gamma_* & \gamma_{j} & \gamma_{j+1} & \cdots  & \gamma_{j+a-2}
\end{pmatrix}
\gamma\big(\lambda^{\trans}\big),
\end{equation}
where $\gamma_*=\overline{\gamma_{j+a-1}+n}\pmod k$. In particular, $\gamma_*=\gamma_{j+a-1}$ if $k\mid n$.
\end{lem}

\begin{proof} Let $\xi$ denote the rim hook $\lambda^{+}/\lambda$. Note that $[\xi^{\trans}]$ has $a$ rows. Suppose $[\xi^{\trans}]$ meets $[{\lambda^{+}}^{\trans}]$ from row $j$ to row $j+a-1$.
Let $\beta(\lambda^{\trans})=(\beta_1,\dots,\beta_{k-1})$. We observe that
\begin{equation} \label{eqn:lambda-plus-transpose}
{\lambda^{+}}^{\trans}_i=
\begin{cases}
\lambda^{\trans}_i &  1\le i\le j-1\\[1ex]
\lambda^{\trans}_{i-1}+1 &  j+1\le i\le j+a-1\\[1ex]
\lambda^{\trans}_i & j+a\le i\le k-1.
\end{cases}
\end{equation}
Since $|\lambda^{+}|=|\lambda|+n$, we have
\begin{equation} \label{eqn:lambda-plus_j}
{\lambda^{+}}^{\trans}_{j}=\lambda^{\trans}_{j+a-1}+n-a+1.
\end{equation}
Let $\beta\big({\lambda^{+}}^{\trans}\big)=\big(\beta'_1,\dots,\beta'_{k-1}\big)$. By (\ref{eqn:beta-gamma}), (\ref{eqn:lambda-plus-transpose}) and (\ref{eqn:lambda-plus_j}), we have
\begin{equation} \label{eqn:beta-lambda+c-vector}
\beta'_i=
\begin{cases}
\beta_i &  1\le i\le j-1\\
\beta_{j+a-1}+n &  i=j \\
\beta_{i-1} & j+1\le i\le i+a-1\\
\beta_i & i+a\le j\le k-1.
\end{cases}
\end{equation}
Thus (\ref{eqn:beta-lambda+c}) holds. Creating $\gamma\big({\lambda^{+}}^{\trans}\big)$ from $\beta\big({\lambda^{+}}^{\trans}\big)$, we have
\[
\gamma\big({\lambda^{+}}^{\trans}\big)=(\gamma_1,\dots,\gamma_{j-1},\gamma_*,\gamma_{j},\dots,\gamma_{j+a-2},\gamma_{j+a},\dots,\gamma_{k-1}), 
\]
where $\gamma_*=\overline{\gamma_{j+a-1}+n}\pmod k$.  The result follows.
\end{proof}

For $\lambda,\mu\in\G(k,m)$ and partitions $\lambda^{+}, \mu^{+}$ such that $\lambda^{+}/\lambda$ and $\mu^{+}/\mu$ are rim hooks of size $n$, the following result shows that if $\mu\neq\lambda$ and $\mu^{+}=\lambda^{+}$, the occurrences of $s_{\lambda^{+}}, s_{\mu^{+}}$ in (\ref{eqn:Schur-expansion}) have the same sign if $k$ divides $n$, and the opposite signs otherwise.

\begin{pro} \label{pro:opposite-sign} Let $k\ge 3$. For any $\lambda,\mu\in\G(k,m)$, let $\lambda^{+},\mu^{+}$ be partitions such that $\lambda^{+}/\lambda$ and $\mu^{+}/\mu$ are rim hooks of size $n$ with $\mu^{+}=\lambda^{+}$ and $\lambda^{+}_1<k$. Then either $\mu=\lambda$, or $\mu\neq\lambda$ and the following relation holds:
\[
\frac{(-1)^{\height(\lambda^{+}/\lambda)}\pet_k(\lambda)}{(-1)^{\height(\mu^{+}/\mu)}\pet_k(\mu)} =
\begin{cases}
1  &  \mbox{if $k\mid n$}\\
-1 &  \mbox{if $k\nmid n$.}\\
\end{cases}
\]
\end{pro}

\begin{proof} Let $d$ be the remainder of $m$ divided by $k$. Let $\beta(\lambda^{\trans})=(\beta_1,\dots,\beta_{k-1})$, and let $\gamma(\lambda^{\trans})=(\gamma_1,\dots,\gamma_{k-1})$. Both of the underlying sets of $\gamma(\lambda^{\trans})$ and $\gamma(\mu^{\trans})$ are $\{1,\dots,k\}\setminus\{k-d\}$ since $\lambda, \mu\in\G(k,m)$.
Suppose $[\lambda^{+}/\lambda]$ ($[\mu^{+}/\mu]$, respectively) has $a$ ($b$, respectively) columns.
By Lemma \ref{lem:cyclic-shift}, for some $j\in\{1,\dots,k-a\}$ we have
\begin{equation} \label{eqn:gamma-lambda+c}
\gamma\big({\lambda^{+}}^{\trans}\big)=
\begin{pmatrix}
\gamma_j &  \gamma_{j+1} &  \gamma_{j+2}  & \cdots & \gamma_{j+a-1} \\
\gamma_* & \gamma_{j}  & \gamma_{j+1} & \cdots  & \gamma_{j+a-2}
\end{pmatrix}
\gamma\big(\lambda^{\trans}\big)
\end{equation}
for some $\gamma_*\in\{1,\dots,k\}$. Let $\gamma\big({\lambda^{+}}^{\trans}\big)=(\gamma'_1,\dots,\gamma'_{k-1})$ be the resulting sequence. Since $\mu^{+}=\lambda^{+}$, the partition $\mu$ is obtained from $\lambda^{+}$ by removing a rim hook of size $n$. By Lemma \ref{lem:non-inversion}, the sum of the elements of $\beta\big(\mu^{\trans}\big)$ equals $\beta_1+\cdots+\beta_{k-1}$, and the sequence $\gamma\big(\mu^{\trans}\big)$ is obtained by acting a $b$-cycle on $\gamma\big({\lambda^{+}}^{\trans}\big)$ possibly with replacement.

(i) If $k\mid n$ then $\gamma_*=\gamma_{j+a-1}$. For some $i\in\{1,\dots,k-b\}$, we have
\begin{equation} \label{eqn:gamma-mu_c}
\gamma\big(\mu^{\trans}\big) =
\begin{pmatrix}
\gamma'_i  & \gamma'_{i+1} & \cdots  & \gamma'_{i+b-2} & \gamma'_{i+b-1} \\
\gamma'_{i+1}  & \gamma'_{i+2} &\cdots & \gamma'_{i+b-1}  & \gamma'_i
\end{pmatrix}
\gamma\big({\lambda^{+}}^{\trans}\big). 
\end{equation} 
By the two cycles in (\ref{eqn:gamma-lambda+c}) and (\ref{eqn:gamma-mu_c}), we have
\[
\ninv\big(\gamma(\lambda^{\trans})\big)+\ninv\big(\gamma(\mu^{\trans})\big)\equiv (a-1)+(b-1)\pmod 2.
\]
Using the facts $\height(\lambda^{+}/\lambda)=n-a$ and $\height(\mu^{+}/\mu)=n-b$, we have
\[
\frac{(-1)^{\height(\lambda^{+}/\lambda)}\pet_k(\lambda)}{(-1)^{\height(\mu^{+}/\mu)}\pet_k(\mu)} =\frac{(-1)^{n-a+\beta_1+\cdots+\beta_{k-1}+\ninv(\gamma(\lambda^{\trans}))+\gamma_1+\cdots+\gamma_{k-1}}}{(-1)^{n-b+\beta_1+\cdots+\beta_{k-1}+\ninv(\gamma(\mu^{\trans}))+\gamma_1+\cdots+\gamma_{k-1}}}=1.
\]

(ii) If $k\nmid n$ then $\gamma_*\in\{1,\dots,k\}\setminus\{\gamma_{j+a-1}\}$. There are two cases.

Case 1. If $\gamma_*=k-d$ then $\gamma\big({\lambda^{+}}^{\trans}\big)$ consists of the $k-1$ elements $\{1,\dots,k\}\setminus\{\gamma_{j+a-1}\}$. Since the underlying set of $\gamma\big(\mu^{\trans}\big)$ is $\{1,\dots,k\}\setminus\{k-d\}$, we observe that this is accomplished by the inverse operation of (\ref{eqn:gamma-lambda+c}), which substitutes $\gamma_{j+a-1}$ back for the element $k-d$ in $\gamma\big({\lambda^{+}}^{\trans}\big)$. Hence $\mu=\lambda$.

Case 2. If $\gamma_*\neq k-d$ then $\gamma_*\in\{\gamma_1,\dots,\gamma_{k-1}\}\setminus\{\gamma_{j+a-1}\}$, say $\gamma_*=t$.  By (\ref{eqn:gamma-lambda+c}), note that $\gamma'_j$ is the newly added $t$ in $\gamma\big({\lambda^{+}}^{\trans}\big)$. Let $\gamma'_i$ be the original $t$. If $\mu\neq\lambda$, we observe that the requested operation on $\gamma\big({\lambda^{+}}^{\trans}\big)$ substitutes $\gamma_{j+a-1}$ for the element $\gamma'_i$. Therefore, we have
\begin{equation} \label{eqn:gamma-mu_c-case2} 
\gamma\big(\mu^{\trans}\big) =
\begin{pmatrix}
\gamma'_i &  \gamma'_{i+1} & \cdots  &\gamma'_{i+b-2} & \gamma'_{i+b-1} \\
\gamma'_{i+1} & \gamma'_{i+2} & \cdots & \gamma'_{i+b-1}  & \gamma_{j+a-1}
\end{pmatrix}
\gamma\big({\lambda^{+}}^{\trans}\big).
\end{equation} 
In this case, the result of (\ref{eqn:gamma-lambda+c}) and (\ref{eqn:gamma-mu_c-case2}) is equivalent to the composition of the following cycles. Let $(\gamma''_1,\dots,\gamma''_{k-1})$ be the sequence resulting from 
\begin{equation} \label{eqn:gamma-lambda+c-case2}
\begin{pmatrix}
\gamma_j &  \gamma_{j+1} &  \gamma_{j+2}  & \cdots & \gamma_{j+a-1} \\
\gamma_{j+a-1} & \gamma_{j}  & \gamma_{j+1} & \cdots  & \gamma_{j+a-2}
\end{pmatrix}
\gamma\big(\lambda^{\trans}\big),
\end{equation}
where $\gamma''_j=\gamma_{j+a-1}$ and $\gamma''_i=t$. Let 
\begin{equation} \label{eqn:interchange}
\gamma\big(\mu^{\trans}\big) =\begin{pmatrix}
t &  \gamma_{j+a-1} \\
\gamma_{j+a-1} & t 
\end{pmatrix}
\begin{pmatrix}
\gamma''_i  & \gamma''_{i+1} & \cdots  &\gamma''_{i+b-2} & \gamma''_{i+b-1} \\
\gamma''_{i+1}  & \gamma''_{i+2} &\cdots & \gamma''_{i+b-1}  & \gamma''_i
\end{pmatrix}
(\gamma''_1,\dots,\gamma''_{k-1}).
\end{equation}
By the cycles in (\ref{eqn:gamma-lambda+c-case2}) and (\ref{eqn:interchange}), we have
\[
\ninv\big(\gamma(\lambda^{\trans})\big)+\ninv\big(\gamma(\mu^{\trans})\big)\equiv (a-1)+(b-1)+1\pmod 2.
\]
It follows that
\[
\frac{(-1)^{\height(\lambda^{+}/\lambda)}\pet_k(\lambda)}{(-1)^{\height(\mu^{+}/\mu)}\pet_k(\mu)} =\frac{(-1)^{n-a+\beta_1+\cdots+\beta_{k-1}+\ninv(\gamma(\lambda^{\trans}))+\gamma_1+\cdots+\gamma_{k-1}}}{(-1)^{n-b+\beta_1+\cdots+\beta_{k-1}+\ninv(\gamma(\mu^{\trans}))+\gamma_1+\cdots+\gamma_{k-1}}}=-1.
\]
Thus the proof is completed.
\end{proof}

\begin{exa} \label{exa:vanish} {\rm
(i) Consider the expansion of $G(3,5)\cdot p_3$ in the Schur basis. Let $\lambda=(2,2,1)$ and $\mu=(2,1,1,1)$. The diagrams of the partitions $\mu^{+}=\lambda^{+}=(2,2,2,2)$ with the rim hooks $\lambda^{+}/\lambda$ and $\mu^{+}/\mu$ are shown on the left of Figure \ref{fig:vanish}. Note that $\gamma(\lambda^{\trans})=(2,3)$, $\gamma\big({\lambda^{+}}^{\trans}\big)=(3,2)$, and $\gamma(\mu^{\trans})=(3,2)$. Moreover, $\pet_3(\lambda)=1$, $\pet_3(\mu)=-1$, $(-1)^{\height(\lambda^{+}/\lambda)}=-1$ and $(-1)^{\height(\mu^{+}/\mu)}=1$. The coefficient of $s_{(2,2,2,2)}$ in the expansion is $-2$.

(ii) Consider the expansion of $G(5,8)\cdot p_3$ in the Schur basis. Let $\lambda=(4,4)$ and $\mu=(3,2,1,1,1)$. The diagrams of the partitions $\mu^{+}=\lambda^{+}=(4,4,1,1,1)$ with the rim hooks $\lambda^{+}/\lambda$ and $\mu^{+}/\mu$ are shown on the right of  Figure \ref{fig:vanish}. Note that $\gamma(\lambda^{\trans})=(1,5,4,3)$, $\gamma\big({\lambda^{+}}^{\trans}\big)=(4,5,4,3)$, and $\gamma(\mu^{\trans})=(4,5,3,1)$. Moreover, $\pet_5(\lambda)=\pet_5(\mu)=1$,  $(-1)^{\height(\lambda^{+}/\lambda)}=1$ and $(-1)^{\height(\mu^{+}/\mu)}=-1$. The term $s_{(4,4,1,1,1)}$ vanishes in the expansion.

}
\end{exa}

\begin{figure}[ht]
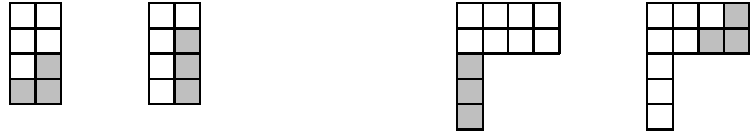

\begin{center}
\begin{tabular}{ccc}   
\ytableausetup{smalltableaux}
\ydiagram{2,2,2,2}*[*(lightgray)]{0,0,1+1,0+2}\hspace{0.4in}
\ydiagram{2,2,2,2}*[*(lightgray)]{0,1+1,1+1,1+1}    & \hspace{1in}  &
\ytableausetup{smalltableaux}
\ydiagram{4,4,1,1,1}*[*(lightgray)]{0,0,1,1,1}\hspace{0.4in}
\ydiagram{4,4,1,1,1}*[*(lightgray)]{3+1,2+2,0,0,0} \end{tabular}
\end{center}
\caption{\small On the left (right, respectively) are the partitions $\mu^{+}=\lambda^{+}=(2,2,2,2)$ ($(4,4,1,1,1)$, respectively) and the rim hooks $\lambda^{+}/\lambda$ and $\mu^{+}/\mu$ of size 3 in Example \ref{exa:vanish}.}
\label{fig:vanish}
\end{figure}

\smallskip
Now we prove Theorem \ref{thm:main-2}.

\medskip
\noindent
\emph{Proof of Theorem \ref{thm:main-2}.} We first show the sufficiency. Given a positive integer  $n$, let $k\ge 3$, $k\mid n$, and $m\ge n$. Using (\ref{eqn:Schur-expansion}),  we shall prove that the expansion of $G(k,m)\cdot p_n$ is not signed multiplicity free by constructing $\lambda,\mu\in\G(k,m)$, $\mu\neq\lambda$, and partitions $\lambda^{+},\mu^{+}$ such that $\lambda^{+}/\lambda$ and $\mu^{+}/\mu$ are rim hooks of size $n$ with $\mu^{+}=\lambda^{+}$. By Proposition \ref{pro:opposite-sign}, we have
\[
(-1)^{\height(\lambda^{+}/\lambda)}\pet_k(\lambda)=(-1)^{\height(\mu^{+}/\mu)}\pet_k(\mu).
\]
So, such occurrences of $s_{\lambda^{+}}$ and $s_{\mu^{+}}$ in (\ref{eqn:Schur-expansion}) always have the same sign.

Let $d$ be the remainder of $m$ divided by $k$, and let $r$ be the remainder of $m-d$ divided by $n$. Notice that $r$ is a multiple of $k$ since $k$ is a common divisor of $n$ and $m-d$. Let $q=1+\lfloor\frac{m-d}{n}\rfloor$. Note that $q\ge 2$. The partitions $\lambda^{+}, \lambda$ and $\mu$ are constructed in the following four cases.

(1) $d\neq 1$ and $q$ is even, say $q=2t$. Let $\lambda^{+}=(d,2^{n\cdot t},1^r)$ and $\lambda=(d,2^{n(t-1)},1^{n+r})$, where the first part $d$ is omitted if $d=0$. Note that the Young diagram $[\lambda]$ is obtained by removing $n$ cells from the second column of $[\lambda^{+}]$. The other partition is $\mu=(d,2^{n(t-1)+r+1},1^{n-r-2})$, where $[\mu]$ is obtained by removing $r+1$ ($n-r-1$, respectively) cells from the first (second, respectively) column of $[\lambda^{+}]$. Notice that $n-r\ge 3$ since $k\mid (n-r)$ and $k\ge 3$.

(2) $d=1$ and $q$ is even, say $q=2t$. Let $\lambda^{+}=(2^{n\cdot t},1^{r+1})$ and  $\lambda=(2^{n(t-1)},1^{n+r+1})$. The other partition is $\mu=(2^{n(t-1)+r+2},1^{n-r-3})$, where $[\mu]$ is obtained by removing $r+2$ ($n-r-2$, respectively) cells from the first (second, respectively) column of $[\lambda^{+}]$. 

(3) $d\neq 1$ and $q$ is odd, say $q=2t+1$. 

\begin{itemize}
\item $r=0$. Let $\lambda^{+}=(d,2^{n\cdot t},1^n)$, $\lambda=(d,2^{n(t-1)},1^{2n})$ and $\mu=(d,2^{n\cdot t})$.
\item $r\neq 0$. Let $\lambda^{+}=(d,2^{n\cdot t+r},1^{n-r})$ and $\lambda=(d,2^{n(t-1)+r},1^{2n-r})$. The other partition is $\mu=(d,2^{nt+1},1^{r-2})$, where $[\mu]$ is obtained by removing $n-r+1$ ($r-1$, respectively) cells from the first (second, respectively) column of $[\lambda^{+}]$. 
\end{itemize}

(4) $d=1$ and $q$ is odd, say $q=2t+1$. 
\begin{itemize}
\item $r=0$. Let $\lambda^{+}=(2^{n\cdot t},1^{n+1})$, $\lambda=(2^{n(t-1)},1^{2n+1})$ and $\mu=(2^{n\cdot t},1)$.
\item $r\neq 0$. Let $\lambda^{+}=(2^{nt+r},1^{n-r+1})$ and $\lambda=(2^{n(t-1)+r},1^{2n-r+1})$. The other partition is $\mu=(2^{nt+2},1^{r-3})$, where $[\mu]$ is obtained by removing $n-r+2$ ($r-2$, respectively) cells from the first (second, respectively) column of $[\lambda^{+}]$. 
\end{itemize}

The necessity is proved as follows. For any $\lambda,\mu\in\G(k,m)$, let $\lambda^{+},\mu^{+}$ be partitions such that $\lambda^{+}/\lambda$ and $\mu^{+}/\mu$ are rim hooks of size $n$ with $\mu^{+}=\lambda^{+}$. We shall prove that $s_{\lambda^{+}}$ either occurs exactly once or vanishes in (\ref{eqn:Schur-expansion}) under the following conditions on $k$ and $m$.

(i) $k\le 2$. Notice that $G(1,m)=s_{(0)}$ and $G(2,m)=s_{(1^m)}$. By the  rule (\ref{eqn:MN-rule}), the expansion of $G(k,m)\cdot p_n$ in the Schur basis is signed multiplicity free for all $n\ge 1$.

(ii) $k\ge 3$, $k\mid n$, and $m<n$. 
If $\mu\neq\lambda$ then there are two distinct cells $c,c'$ in the Young diagram $[\lambda^{+}]$ whose rim hooks are $\lambda^{+}/\lambda$ and $\mu^{+}/\mu$, respectively (e.g. see Figure \ref{fig:hook-length}). Since the hook length of $c,c'$ is $n$, we have $|\lambda^{+}|\ge 2n$, which is against $m<n$. Hence $\mu=\lambda$.

(iii) $k\ge 3$ and $k\nmid n$. By Proposition \ref{pro:k=2}, if $\lambda^{+}_1\ge k$ then $\mu=\lambda$. Let $\lambda^{+}_1< k$. By Proposition \ref{pro:opposite-sign}, either $\mu=\lambda$, or $\mu\neq\lambda$ and
\[
 (-1)^{\height(\lambda^{+}/\lambda)}\pet_k(\lambda)+(-1)^{\height(\mu^{+}/\mu)}\pet_k(\mu)=0.
\]
Hence $s_{\lambda^{+}}$ either occurs once or vanishes.  The proof of Theorem \ref{thm:main-2} is completed.
\qed

\begin{figure}[ht]
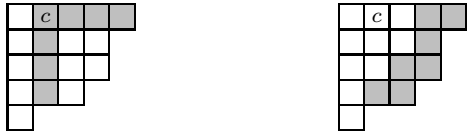

\begin{center}

\ytableaushort
{\none {c}, \none, \none, \none, \none}
* {5,4,4,3,1}
* [*(lightgray)]{1+4,1+1,1+1,1+1,1+0}
\hspace{1in}
\ytableaushort
{\none {c}, \none, \none, \none, \none}
* {5,4,4,3,1}
* [*(lightgray)]{3+2,3+1,2+2,1+2,1+0}

\end{center}
\caption{\small The hook and rim hook of a cell $c$ of a Young diagram.}
\label{fig:hook-length}
\end{figure}

\section{Concluding remarks}

Grinberg \cite[Corollary~2.18]{Grinberg} generalized Theorem \ref{thm:k-Petrie-Schur} to the product of  $G(k,m)$ with a Schur function $s_{\mu}$, and proved that the expansion of $G(k,m)\cdot s_{\mu}$ in the Schur basis is signed multiplicity free, i.e.,
\begin{equation} \label{eqn:generalized-k-Petrie}
G(k,m)\cdot s_{\mu} =\sum_{\lambda\,\vdash\, m+|\mu|}\pet_k(\lambda,\mu)s_{\lambda}, 
\end{equation}
where $
\pet_k(\lambda,\mu):=\det\big[\chi(0\le\lambda_i-\mu_j-i+j<k  )\big]_{i,j=1}^{\ell}$ (The number $\ell$ can be chosen so that $\ell(\lambda)\le \ell$ and $\ell(\mu)\le \ell$ by adding zero parts to $\lambda$ and $\mu$ \cite[Lemma~2.7]{Grinberg}). The $k$-Petrie number $\pet_k(\lambda)$ is simply $\pet_k(\lambda,\varnothing)$. We are interested in a combinatorial interpretation of the coefficients of the expansion (\ref{eqn:generalized-k-Petrie}). 

In the spirit of Jacobi--Trudi identity \cite[I(3.4)]{Mac}, Walker \cite{Walker} defined the \emph{modular Schur function} for a partition $\lambda$ by
\begin{equation} \label{eqn:modular-Schur-function}
G(k,\lambda):=\det\big[G(k,\lambda_i-i+j)\big]_{i,j=1}^{\ell(\lambda)}.
\end{equation}
Walker studied the transition matrix $A=(a_{\lambda,\mu})$, with rows and columns indexed by a set of partitions, defined by 
\[
G(k,\lambda)=\sum_{\mu} a_{\lambda,\mu} s_{\mu},
\]
and proved that the transition matrix $A$ is the direct sum of submatrices indexed by the $k$-cores of partitions $\lambda\vdash m$ \cite[Corollary~2.8]{Walker}. For example, in the case $k=3$ and $m=4$, the partitions $\lambda\vdash 4$ with $\core_3(\lambda)=(1)$ are $\{(4), (2,2), (1,1,1,1)\}$, and there are another two $3$-cores, $(3,1)$ and $(2,1,1)$. Indexed by the $3$-cores, the transition is expressed as follows:
\[
\begin{bmatrix}
G(3,(4)) \\
G(3,(2,2)) \\
G(3,(1,1,1,1)) \\
G(3,(3,1)) \\
G(3,(2,1,1))
\end{bmatrix}
=
\left[\begin{array}{ccc|c|c}
0  &  1  & -1 &  0  &  0 \\ 
1  &  0  &  1 &  0  &  0 \\
-1 &  1  &  0 &  0  &  0 \\ \hline
0  &  0  &  0 &  1  &  0 \\ \hline
0  &  0  &  0 &  0  &  1
\end{array}
\right]
\begin{bmatrix}
s_{(4)} \\
s_{(2,2)} \\
s_{(1,1,1,1)} \\
s_{(3,1)} \\
s_{(2,1,1)}
\end{bmatrix}.
\]
Our first main result gives a combinatorial interpretation for the entries in the first row of the transition matrix. We are interested in a combinatorial method to complete the transition matrix by expanding $G(k,\lambda)$ in the Schur basis combinatorially.


\begin{thebibliography}{99}

\bibitem{Doty-Walker} S. Doty, G. Walker, Modular symmetric functions and irreducible modular representations of general linear groups, J. Pure Appl. Algebra 82(1) (1992), 1--26.

\bibitem{FM} H. Fu, Z. Mei, Truncated homogeneous symmetric functions, Lin. Multilin. Algebra 70(3) (2022), 438--448.

\bibitem{GW} M. Gordon, E.M. Wilkinson, Determinants of Petrie matrices, Pacific J. Math. 51 (1974), 451--453.

\bibitem{Grinberg} D. Grinberg, Petrie symmetric functions, {\tt arXiv:2004.11194v3} [math.CO], 2021.

\bibitem{JK} G. James, A. Kerber, The representation theory of the symmetric group, Encyclopedia of Mathematics and its Applications, vol. 16, Addison-Wesley Publishing Co., Reading, MA, 1981.

\bibitem{LP} L. Liu, P. Polo, On the cohomology of line bundles over certain flag schemes II, J. Combin. Theory, Ser. A 178 (2021), 105352.

\bibitem{Mac} I.G. Macdonald, Symmetric functions and Hall polynomials, Oxford Classic Texts in the Physical Sciences, Clarendon Press, Oxford University Press, New York, second ed., 2015.




\bibitem{EC2} R.P. Stanley, Enumerative Combinatorics, vol. 2, Cambridge University Press, New York/Cambridge, 1996.

\bibitem{Walker} G. Walker, Modular Schur functions, Trans. Amer. Math. Soc. 346(2) (1994), 569--604.

\end{thebibliography}
\end{document}